\documentclass[12pt, twoside, leqno]{article}

\usepackage{amsmath,amsthm}
\usepackage{amssymb}

\usepackage{enumerate}

\usepackage{graphicx}

\newtheorem{thm}{Theorem}[section]

\newtheorem{lem}[thm]{Lemma}
\newtheorem{prop}[thm]{Proposition}

\theoremstyle{definition}
\newtheorem{defn}[thm]{Definition}
\newtheorem{rem}[thm]{Remark}
\newtheorem{Exa}[thm]{Example}

\numberwithin{equation}{section}

\usepackage{pgf,tikz}
\usetikzlibrary{arrows}

\begin{document}

\baselineskip=13pt

%%%%%%%%%%%%%%%%

\title{On the automorphism groups of connected bipartite irreducible graphs }

\author{S.Morteza Mirafzal\\
Department of Mathematics \\
  Lorestan University, Khorramabad, Iran\\
\\
E-mail:smortezamirafzal@yahoo.com\\
E-mail: mirafzal.m@lu.ac.ir}

\date{}

\maketitle

%% Classification and key words; note that the 2010 classification is used:05C25, 05C69, 94C15

\renewcommand{\thefootnote}{}



\footnote{\emph{Date}: }

\renewcommand{\thefootnote}{\arabic{footnote}}
\setcounter{footnote}{0}
\date{}
%----------additions
%%%
\begin{abstract}
Let $G=(V,E)$ be a  graph with the vertex-set $V$ and  the edge-set $E$.
 Let $N(v)$ denote  the set of neighbors of the vertex $v$ of $G.$ 
 The graph  $G$ is  called    $ irreducible $      whenever for every 
     $v,w \in V$ if $v \neq w$,   then   $N(v)\neq N(w).$  In this paper, we present a method for finding   automorphism groups of connected bipartite irreducible graphs.  Then,   by our method, we 
 determine   automorphism groups of some classes of connected
 bipartite irreducible  graphs, including a class of graphs which are derived from Grassmann graphs. Let $a_0$ be a fixed positive integer.    We show that if $G$ is a connected non-bipartite irreducible   graph  such that 
   $c(v,w)=|N(v)\cap N(w)|=a_0$ when $v,w$ are adjacent, whereas $c(v,w) \neq a_0$, when $v,w$ are not adjacent, 
    then $G$ is a $stable$ graph, that is,  the automorphism group
 of the bipartite double cover  of $G$ is isomorphic with the group $Aut(G) \times \mathbb{Z}_2$.   Finally, we show that the Johnson graph $J(n,k)$ is a stable graph.
\end{abstract}
\maketitle

\section{Introduction} In this paper, a graph $G=(V,E)$ is
considered as an undirected simple finite graph,  where $V=V(G)$ is the vertex-set
and $E=E(G)$ is the edge-set. For   the terminology and notation
not defined here, we follow $[1,2,4,7]$.\

Let $G=(U \cup W,E), $ $U \cap W= \emptyset$ be a bipartite graph with parts $U$ and $W$. It is quite possible that   we wish to  construct    some other graphs which are related to  $G$ in some aspects. For instance  there are cases in which we can construct a graph $G_1=(U,E_1)$ such that we have $Aut(G) \cong Aut(G_1)$, where $Aut(X)$ is the automorphism group  of the graph $X$. For example note   the following cases.\

(i)  Let $ n \geq 3 $ be an integer and $ [n] = \{1,2,..., n \}$. Let $ k$ be an integer such that
$1\leq k <\frac{n}{2}$. The graph $B(n,k)$   introduced in [16] is a graph with the vertex-set  $V=\{v \  |  \ v \subset [n] ,  | v |  \in \{ k,k+1  \} \} $ and the
edge-set $ E= \{ \{ v , w \} \  | \  v , w \in V , v \subset w $ or $ w \subset v \} $. It is clear that the graph $B(n,k)$ is a bipartite graph with the vertex-set $V=V_1 \cup V_2$, where $V_1=\{ v \subset [n]  \  |  \  |v| =k \}$ and $V_2=\{ v \subset [n] \  | \  |v| =k+1 \}$. This graph has some interesting properties which have been  investigated recently [11,16,17,20]. Let $G=B(n,k)$ and let  $G_1=(V_1,E_1)$ be the Johnson graph $J(n,k)$ which can be constructed on the vertex-set $V_1$. It has been proved that if $n\neq 2k+1$, then $Aut(G) \cong Aut(G_1)$, and if $n=2k+1$, then $ Aut(G) \cong Aut(G_1) \times \mathbb{Z}_2$ [16].\

(ii) Let $n$ and $k$ be integers with  $n>2k,  k\geq1$. Let  $V$ be the set of all $k$-subsets
and $(n-k)$-subsets of $[n]$.
 The $bipartite\  Kneser\  graph$ $H(n, k)$ has
$V$ as its vertex-set, and two vertices $v,  w$ are adjacent if and only if $v \subset w$ or $w\subset v$. It is clear that $H(n, k)$ is a bipartite graph.
In fact,
if  $V_1=\{ v \subset [n] \   |  \  |v| =k \}$ and $V_2=\{ v \subset [n] \  | \  |v| =n-k \}$, then $\{ V_1, V_2\}$
is a partition of $V(H(n ,k))$ and every edge of $H(n, k)$ has a vertex in $V_1$ and a  vertex in $V_2$ and
$| V_1 |=| V_2 |$. Let $G=H(n,k)$ and let $G_1=(V_1,E_1)$ be the Johnson graph $J(n,k)$ which can be constructed on the vertex-set $V_1$. 
It has been proved that  $ Aut(G) \cong Aut(G_1) \times \mathbb{Z}_2$ [18].\

(iii) Let $n, k$ and $l$ be integers with $0 < k < l < n $. The $set$-$inclusion$ 
$graph$ $G(n, k, l)$ is the graph whose vertex-set consists of all $k$-subsets and $l$-subsets of $[n] $, where two distinct vertices are adjacent if one of them is contained
in another. It is clear that the graph $G(n, k, l)$ is a bipartite graph with the vertex-set $V=V_1 \cup V_2$, where $V_1=\{ v \subset [n] \   |  \  |v| =k \}$ and $V_2=\{ v \subset [n] \  | \  |v| =l \}$. It is easy to show that $G(n, k, l) \cong G(n,n-k,n-l)$, hence we assume that $k+l \leq n$. It is clear that if $l=k+1$, then $G(n, k, l)=B(n,k)$, where $B(n,k)$ is the graph which is defined in (i). Also,  if $l=n-k$,  then $G(n, k, l)=H(n,k)$,  where $H(n,k)$ is the graph which is introduced in (ii).  Let $G=G(n, k, l)$ and let  $G_1=(V_1,E_1)$ be the Johnson graph $J(n,k)$ which can be constructed on the vertex-set $V_1$. It has been proved  that if $n\neq k+l$, then $Aut(G) \cong Aut(G_1)$, and if $n=k+l$, then $ Aut(G) \cong Aut(G_1) \times \mathbb{Z}_2$ [9].\
 
Let $G=(V,E)$ be a graph. The $bipartite\  double\  cover$ of $G$ which we denote it  by $B(G)$  is a graph with the vertex-set $V \times \{ 0,1 \}$,     in which vertices $(v,a)$ and $(w,b)$ are adjacent if and only if 
 $a \neq b$ and $\{v,w\} \in E$. A graph $G$ is called $stable$ if and only if $Aut(B(G)) \cong Aut(G) \times \mathbb{Z}_2 $.\

In this paper,  we generalize the results of our examples to some other classes of bipartite graphs. In fact, we state some accessible conditions such that if for a bipartite graph $G=(V,E)=(U\cup W,E)$ these conditions hold, then we can determine the automorphism group of the graph $G$. Also,  we determine the automorphism group of a class of graphs which are derived from Grassmann graphs. In particular, we determine   automorphism groups   of   bipartite double covers  of some classes of graphs. In fact,  we show that if $G$ is a non-bipartite connected irreducible graph, and $a_0$ is a positive integer such that  
$c(v,w)=|N(v)\cap N(w)|=a_0$, when $v$ and $w$ are adjacent, whereas $c(v,w) \neq a_0 $ when $v$ and $w$ are not adjacent,  
   then  the graph $G$ is a stable graph. Finally, we show that Johnson graphs are stable graphs. 

\section{Preliminaries}
The graphs $G_1 = (V_1,E_1)$ and $G_2 =
(V_2,E_2)$ are called $isomorphic$, if there is a bijection $\alpha
: V_1 \longrightarrow V_2 $   such that  $\{a,b\} \in E_1$ if and
only if $\{\alpha(a),\alpha(b)\} \in E_2$ for all $a,b \in V_1$.
In such a case the bijection $\alpha$ is called an $isomorphism$.
An $automorphism$ of a graph $G $ is an isomorphism of $G
$ with itself. The set of automorphisms of $\Gamma$  with the
operation of composition of functions is a group  called the
$automorphism\  group$ of $G$ and denoted by $ Aut(G)$.

 The
group of all permutations of a set $V$ is denoted by $Sym(V)$  or
just $Sym(n)$ when $|V| =n $. A $permutation$ $group$ $\Gamma$ on
$V$ is a subgroup of $Sym(V).$  In this case we say that $\Gamma$ $acts$
on $V$. If $\Gamma$ acts on $V$ we say that $\Gamma$ is
$transitive$ on $V$ (or $\Gamma$ acts $transitively$ on $V$), when there is just
one orbit. This means that given any two elements $u$ and $v$ of
$V$, there is an element $ \beta $ of  $G$ such that  $\beta (u)= v
$.  If $X$ is a graph with vertex-set $V$  then we can view
each automorphism of $X$ as a permutation on $V$  and so $Aut(X) = \Gamma$ is a
permutation group on $V$.

A graph $G$ is called $vertex$-$transitive$ if  $Aut(G)$
acts transitively on $V(\Gamma)$. We say that $G$ is $edge$-$transitive$ if the group $Aut(G)$ acts transitively  on the edge set $E$, namely, for any $\{x, y\} ,   \{v, w\} \in E(G)$, there is some $\pi$ in $Aut(G)$,  such that $\pi(\{x, y\}) = \{v, w\}$.  We say that $G$ is $symmetric$ (or $arc$-$transitive$) if  for all vertices $u, v, x, y$ of $G$ such that $u$ and $v$ are adjacent, and also, $x$ and $y$ are adjacent, there is an automorphism $\pi$ in $Aut(G)$ such that $\pi(u)=x$ and $\pi(v)=y$. We say that $G$ is $distance$-$transitive$ if  for all vertices $u, v, x, y$ of $G$ such that $d(u, v)=d(x, y)$, where $d(u, v)$ denotes the distance between the vertices $u$ and $v$  in $G$,  there is an automorphism $\pi$ in $Aut(\Gamma)$ such that  $\pi(u)=x$ and $\pi(v)=y.$\\
Let  $n,k \in \mathbb{ N}$ with $ k < n,   $ and let $[n]=\{1,...,n\}$. The $Johnson\  graph$ $J(n,k)$ is defined as the graph whose vertex set is $V=\{v\mid v\subseteq [n], |v|=k\}$ and two vertices $v$,$w  $ are adjacent if and only if $|v\cap w|=k-1$.   The Johnson\ graph $J(n,k)$ is a  distance-transitive graph  [2].   It is   easy   to show that the set  $H= \{ f_\theta \ | \  \theta \in$ S$ym([n]) \} $,  $f_\theta (\{x_1, ..., x_k \}) = \{ \theta (x_1), ..., \theta (x_k) \} $,    is a subgroup of $ Aut( J(n,k) ). $   It has been shown that   $Aut(J(n,k)) \cong$ S$ym([n])$, if  $ n\neq 2k, $  and $Aut(J(n,k)) \cong$ S$ym([n]) \times \mathbb{Z}_2$, if $ n=2k$,   where $\mathbb{Z}_2$ is the cyclic group of order 2 [10,19].

The group $\Gamma$ is called a semidirect product of $ N $ by $Q$,
denoted by $ \Gamma=N \rtimes Q $,
 if $\Gamma$ contains subgroups $ N $ and $ Q $ such that:  (i)$
N \unlhd \Gamma $ ($N$ is a normal subgroup of $\Gamma$); (ii) $ NQ = \Gamma $; and
(iii) $N \cap Q =1 $. \

Although 
in most situations  it is difficult  to determine the automorphism group
of a graph $G$,  there are various papers in the literature   dealing with this,   and some of the recent works
include [5,6,10,14,15,16,18,19,24].

\section{Main Results}

The proof of the  following  lemma  is easy but its result  is necessary for proving  the results of  our work.

\begin{lem}
Let $G= (U \cup W,E)$, $U \cap W=\emptyset $ be a connected bipartite graph.  If $f$ is an automorphism of the graph $G$,  then $ f(U)=U$ and $f(W) =W$,  or $ f(U) = W $ and $f(W) = U$.

\end{lem}

\begin{proof}
Automorphisms of $G$ preserve
distance between vertices and since two vertices are in the same part if and
only if they are at even distance from each other, the result follows.

\end{proof}

 We have  the following   definition due to Sabidussi [22].

\begin{defn}
Let $G=(V,E)$ be a  graph with the vertex-set $V$ and  the edge-set $E$. Let $N(v)$ denote the set of neighbors of the vertex $v$ of $G.$  We say that $G$ is an  irreducible  graph  if for every pare of distinct vertices $x,y \in V$ we have $N(x)\neq N(y)$.
\end{defn}
From Definition 3.2,  it follows that the cycle $C_n$, $n\neq 4$,  is  irreducible, but the complete bipartite graph $K_{m,n}$ is not    irreducible, when $(m,n) \neq (1,1)$.

\begin{lem} 
Let $G= (U \cup W,E)$, $U \cap W=\emptyset $ be a bipartite irreducible  graph. If $f$ is an automorphism of $ G$ such that $f(u)=u$ for every $u\in U$, then $f$ is the identity automorphism of $ G$.

\end{lem}

\begin{proof}
Let $ w\in W$ be an arbitrary vertex. Since $f$ is an automorphism of the graph $G$, then for the set  $N(w)= \{ u |  u\in U, u \leftrightarrow w \}$, we have $f(N(w))= \{ f(u) |  u\in U, u \leftrightarrow w \}=N(f(w))$. On the other hand, since for every $u\in U$, $f(u)=u$, then we have 
$f(N(w))=N(w) $, and therefore $N(f(w))=N(w) $.  Now since $G$ is an irreducible   graph we must have $f(w)=w$. 
 Therefore, for every vertex $x$ in $V(G)$  we have $f(x)=x$ and thus $f$ is the identity automorphism of the graph $G$.

\end{proof}

Let $G= (U \cup W,E)$, $U \cap W=\emptyset $ be a bipartite graph. We can construct various graphs on the set $U$. We show that some of these graphs can help us in finding the automorphism group of the graph $G$. 

\begin{defn}
Let $G= (U \cup W,E)$, $U \cap W=\emptyset $  be a bipartite graph. Let $G_1=(U,E_1)$ be   a  graph with the vertex-set $U  $ such that the following conditions  hold:  \\
(i) every automorphism of the graph $G_1$ can be uniquely extended to an automorphism of  the graph $G$.  
In other words, if $f$ is an automorphism of the graph $G_1$,  then there is a unique  automorphism $e_f$ in the automorphism group of  $G$ such that ${(e_f)|}_U=f$, where ${(e_f)|}_U$ is the
 restriction of the automorphism $ e_f $ to the set $U$.\\
(ii) If $f \in Aut(G)$ is such that $f(U)=U$, then the restriction of $f$ to $U$ is an automorphism of the graph $G_1$. In other words,   if  $f \in Aut(G)$ is such that $f(U)=U$  then $f|_U \in Aut(G_1).$\\
 When such a graph $G_1$ exists, then  we say that the graph $G_1$ is a faithful representation  of $G$.

\end{defn}

\begin{rem} 
 Let $G= (U \cup W,E)$, $U \cap W=\emptyset $ be a bipartite irreducible graph, and $G_1=(U,E_1)$ be a graph. If $f \in Aut(G_1)$ can be extended to an automorphism $g$ of the graph $G$, then $g$ is unique. In fact if $g$ and $h$ are extensions of the automorphism $f \in Aut(G_1)$ to automorphisms of $G$, then $i=gh^{-1}$ is an automorphism of the graph $G$ such that the restriction of $i$ to the set $U$ is the identity automorphism. Hence by Lemma 3.3, the automorphism $i$  is the identity automorphism of the graph $G$, and therefore $g=h$. Hence, according to Definition 3.4,  the graph $G_1$ is a  faithful representation of    the graph $G$  if and only if every automorphism of $G_1$ can be extended to an automorphism of $G$ and every automorphism of $G$ which fixes $U$ setwise is an automorphism of $G_1$.
\end{rem} 

\begin{Exa}
Let $G=H(n,k)=(V_1 \cup V_2, E)$ be the bipartite Kneser
 graph which 
is introduced in (ii) of  the introduction of the present paper. Let $G_1=(V_1,E_1)$ be the Johnson graph which can be constructed on the vertex $V_1$.  
It can be shown that the 
 graph $G_1$ is a faithful representation of $G$ [18].

\end{Exa}

In the next theorem, we show that if 
 $G=(U \cup W,E)$,   $U \cap W= \emptyset $  is a connected  bipartite
  irreducible   graph  with  $G_1=(U,E_1)$ as   
a faithful representation of  $G$, then we can 
determine the automorphism group of the graph $G$, 
provided   the automorphism group of the graph  $G_1 $ has been
 determined.

Let $G=(U \cup W,E)$,    $U \cap W= \emptyset$, 
 be a connected
  bipartite   irreducible    graph   such that  $G_1=(U,E_1)$ is  a 
 faithful representation  of  $G$. 
If $f\in Aut(G_1)$ then we let $e_f$ be its unique   extension  to $Aut(G).$ 
It is easy to see that $E_{G_1}=\{e_f | f \in Aut(G_1)  \}$,  with the operation of composition, 
 is a group. 
 Moreover, it is easy to see that $E_{G_1}$ and $Aut(G_1)$ are isomorphic (as abstract groups). \\
For the bipartite graph $G=(U \cup W,E)$ we let
$S(U)= \{ f \in Aut(G) |  f(U)=U  \}$=${Aut(G)}_U$, the stabilizer subgroup of the set $U$ in the group $Aut(G)$. The next proposition shows that when $G_1=(U,E_1)$ is a   faithful representation  of $G$, then $S(U)$ is a   familiar group.

\begin{prop}Let $G=(U \cup W,E)$,    $U \cap W= \emptyset$  be a connected bipartite irreducible  graph  such that $G_1=(U,E_1)$ is   a  faithful representation  of   $G$. Then  $S(U) \cong Aut(G_1)$, where $S(U)= \{ f \in Aut(G) | f(U)=U  \}.$

\end{prop}

\begin{proof}Let $f$ be an automorphism of the graph
$G_1$. Then by definition of the graph $G_1$ we deduce that $ e_f $ is an automorphism of the graph $G$ such that $ e_f(U)=U $.  Hence, we have $E_{G_1} \leq S(U), $ where $E_{G_1}$ is the group which is defined preceding   this theorem. \\
On the other hand, if $g \in S(U)$, then $g(U)=U$. Thus by the definition of the graph $G_1$,  the restriction of $g$ to $U$ is an automorphism of the graph $G_1$. In other words,    $h=g|_{U} \in Aut(G_1)$. Therefore,  by Definition 3.4,  there is an automorphism $e_h$ of the graph $G$ such that $e_h(u)=g(u)$ for every $u \in U.$ Now by Remark 3.5,  we deduce that $g=e_h \in E_{G_1}$. Hence we have $ S(U) \leq E_{G_1}.$ We now deduce  that $S(U)=E_{G_1}.$ Now, since $E_{G_1} \cong Aut(G_1)$, we conclude that $S(U) \cong Aut(G_1).$

\end{proof}
Let $G=(U \cup W,E)$, $U \cap W= \emptyset$    be a connected bipartite graph.
It is quite possible that   $f(U)=U, $ for every automorphism of the graph $G.$  For example if $|U| \neq |W|$, or $U$ contains a vertex of degree $d$, but $W$ does not contain a vertex of degree $d$, then we have $f(U)=U  $ for every automorphism $f$ of the graph $G.$ In such a case we have $Aut(G)=S(U)$, and hence by Proposition 3.7,   we have the following theorem. 

\begin{thm}Let $G=(U \cup W,E)$,   $U \cap W= \emptyset$ be a connected bipartite irreducible  graph such that $G_1=(U,E_1)$ is  a  faithful representation  of   $G$.
 If $Aut(G)=S(U)$,   then  $ Aut(G) \cong Aut(G_1)$. 

\end{thm}
Let $G=(U \cup W,E)$,   $U \cap W= \emptyset$ be a connected bipartite irreducible   graph.
 Concerning the automorphism group of $G$, we can  say even more if $|U|=|W|.$ When 
$|U|=|W|$ then there is a bijection $\theta : U \rightarrow W.$ Then ${\theta}^{-1}\cup \theta= t$  is a permutation on the vertex-set of the  graph $G$ such that $t(U)=W$ and $t(W)=U$.    In the following theorem, we show that if the graph $G$ has   a  faithful representation   $G_1=(U,E_1)$,  and if such a permutation $t$   is an automorphism of the graph $G$,  then the automorphism group of the graph $G$ is a familiar group.

  \begin{thm}Let $G=(U \cup W,E)$, $U \cap W= \emptyset$ be a connected  bipartite irreducible  graph  such that $G_1=(U,E_1)$ is  a  faithful representation  of     $G$ and $|U|=|W|$. Suppose that  there is an automorphism $t$ of the graph $G$ such that
  $t(U)=W$. Then $Aut(G)=Aut(G_1)\rtimes H$  where  $H=\langle t \rangle $ is  the  subgroup generated by  $t$ in the group $Aut(G)$.

  \end{thm}
  
   \begin{proof}   Let  $S(U)= \{ f \in Aut(G) |  f(U)=U  \}.$ It is clear  that $S(U)$ is a subgroup of $Aut(G).$ Let $g \in Aut(G)$
be such that $g(U) \neq U$.  Then by Lemma 3.1,  we have $g(U)=W$, and hence $tg(U)=t(W)=U$. Therefore, 
$tg \in S(U)$, and hence there is an element $h \in S(U)$ such that $tg=h$. Thus, $g=t^{-1}h \in \langle t, S(U)\rangle$, where $\langle t, S(U)\rangle=K$ is the subgroup of $Aut(G)$ which is generated by $t$ and $S(U)$.
It  follows that $Aut(G) \leq K$. Since $K \leq Aut(G)$,    we deduce that $K=Aut(G)$.
If $f$ is an  arbitrary element in the subgroup $S(U)$ of $K$, then we have $(t^{-1}ft)(U)=(t^{-1}f)(W)= t^{-1}(f(W))= (t^{-1})(W)=U$, hence $t^{-1}ft \in S(U).$ 
 We now deduce that $S(U)$ is a normal subgroup of the group $K$. Therefore $K=\langle t, S(U)\rangle=S(U)\rtimes \langle t \rangle=S(U) \rtimes H, $ where
$H=\langle t \rangle$. We have seen in Proposition 3.8. that   $S(U) \cong Aut(G_1)$, and hence we conclude that 
 $K=Aut(G)\cong  Aut(G_1) \rtimes H$.
               
\end{proof}                  
 In  the sequel, we will see how Theorem 3.8,  and Theorem 3.9,  can help us in determining the automorphism groups of some classes of bipartite graphs.
 
 $${\bf Some\  Applications}$$\
 Let $G=(U \cup W)=G(n,k,l)$ be  the bipartite graph which is defined in (iii) of the introduction of the present paper. Then  $U=\{ v \subset [n] \   |  \  |v| =k \}$ and $W=\{ v \subset [n] \  | \  |v| =l \}$.
 It is easy to show that $G$ is  connected and irreducible. Let $G_1=(U,E_1)$ be the Johnson graph which can be constructed on the  set $U$. By a proof exactly similar to what   appeared in [16,18] and later  [9],  it can be shown that $G_1$ is a  faithful representation of   $G$. We know that $Aut(G_1)=H=\{f_{\theta} | \theta \in Sym([n])$\}, where $ f_{\theta}(v)=\{ \theta(x)| x\in v \} $, for every  $v\in U$.  Because $k<l$ and  $k+l\leq n$ imply that $k<\frac{n}{2}$.
      When $k+l=n$,  then the mapping $t: V(G) \rightarrow V(G)$, defined by the rule $t(v)=v^c$, where $v^c$ is the complement of the set $v$ in the set $[n]= \{1,2,3,...,n  \}$,  is an automorphism of $G$. It is clear that  $t(U)=W $ and $t(W)=U$. Moreover,  $t$ is of order 2,  and  hence  $\langle t \rangle \cong \mathbb{Z}_2$. It is easy to   show that if $f \in H$,  then  $ft=tf$ [16,19].   Now, from Theorem 3.8,  and Theorem 3.9,  we   obtain the following theorem which has been given in [9].
 
\begin{thm}

Let $n,k$ and $l$ be integers with  $1 \leq k < l \leq n-1$ and $G=G(n,k,l)$. If $n \neq k+l$, then $Aut(G) \cong$  $Sym([n]) $, and if $n=k+l, $ 
then   $Aut(G)=H \rtimes \langle t \rangle   \cong  H \times \langle t \rangle   \cong Sym([n]) \times \mathbb{Z}_2$,  where $H$ and  $t$ are the group and automorphism which are defined preceding  this theorem.

\end{thm}

                                                                                                                                       We now consider a class of graphs which are in some combinatorial  aspects similar to Johnson graphs.
                                                                                                                                       \begin{defn} Let $p$ be a positive prime integer and $q=p^m$ where $m$ is a positive integer. Let $n,k$ be positive integers with  $k <n$.  Let $V(q,n)$ be a vector space of dimension $n$ over the finite field $\mathbb{F}_q.$                                                                                                                                        Let $V_k$ be the family  of all 
subspaces  of  $V(q,n)$ of dimension $k$. 
Every element of $V_k$ is also called a $k$-subspace. 
The Grassmann graph $G(q,n,k)$ is the graph with the  vertex-set $V_{k}$,  in which two vertices $u$ and $w$ are adjacent if and only if $\dim(u\cap w)=k-1$.\
\end{defn}
Note that if $k = 1$, we have a complete graph, so we shall assume that $k >1 $.
 It is clear that  the number of vertices of the Grassmann graph $G(q,n,k)$, that is, 
  $|V_k|$, is
the Gaussian binomial coefficient,
$${n\brack k}_q= \dfrac{(q^{n}-1)(q^n-q)\cdots (q^{n}-q^{k-1})}{(q^{k}-1)(q^k-q)\cdots (q^k-q^{k-1})} =\dfrac{(q^{n}-1)\cdots (q^{n-k+1}-1)}{(q^{k}-1)\cdots (q-1)}.$$

Noting that ${n\brack k}_q={n\brack n-k}_q$,  it follows that $|V_k|=|V_{n-k}|$. It is easy to show that if $1 \leq i < j \leq \frac{n}{2}$, then $|V_i| < |V_j|$. Let $( , )$ be any nondegenerate symmetric bilinear form
   on $V(q,n)$. For each
    $X \subset V(q,n)$ we let $X^{\perp}=\{ w \in V(q,n) | (x,w)=0, $ for  every   $ x \in X  \}.$ 
    It can be shown that if $v$ is a  subspace of  $V(q,n)$,  then  $v^{\perp}$ is also a subspace of $V(q,n)$ and  $dim(v^{\perp})=n-dim(v)$. It can be shown that 
    $G(n,q,k) \cong G(n,q,n-k)$ [2],  and  hence in the sequel we assume that $k \leq \frac{n}{2}$.

It is easy to see that  the distance between two vertices $v$ and $w$ in this graph is $k-dim(v\cap w)$. 
The Grassmann graph is a distance-regular graph of diameter $k$ [2].        
                                                                                                                                       Let $K$ be a field and $V(n)$ be a vector space of dimension $n$ over the field $K$.  Let  $\tau : K\longrightarrow K$  be a 
field automorphism. A semilinear operator  on  $V(n)$ is a mapping  $f : V(n)\longrightarrow V(n)$ such that, 
 \begin{center}
$f(c_{1}v_{1} + c_{2}v_{2}) = \tau (c_{1})f(v_{1}) + \tau (c_{2})f(v_{2})\  (c_{1}, c_{2}\in K, \ and \  v_{1},   v_{2}\in V(n))$.
\end{center}
A semilinear operator $f : V(n)\longrightarrow V(n)$ is a semilinear automorphism if it is a bijection. 
 Let $\Gamma L_n (K)$ be the group of semilinear automorphisms on $V(n)$. Note that this group contains $A(V(n))$, where $A(V(n))$ is the group of 
 non-singular linear mappings on the space $V(n)$. Also,  this group  contains a normal subgroup isomorphic to $K^{*}$,
  namely,  the group $Z= \{ kI_{V(n)} | k \in K \}$, where $I_{V(n)}$ is the identity mapping on $V(n)$. We denote the quotient group  $\frac{\Gamma L_n  (K)}{Z}$  by 
$P \Gamma L_n (K) $.\

Note that if $(a+Z) \in P \Gamma L_n (K)$ and $x$ is an $m$-subspace of $V(n)$, then $ (a+Z)(x)=\{ a(u) | u \in x \}$ is an $m$-subspace of  $V(n)$. In the sequel, we also denote $(a+Z) \in P \Gamma L_n (K)$ by $a$. 
 Now, if $a \in P \Gamma L_n (\mathbb{F}_q)$,  it is easy to see that  the mapping 
  $f_a : V_k \longrightarrow V_k$, defined by   the  rule $f_a(v)=a(v)$,  
  is an automorphism of the Grassmann graph $G=G(q,n,k)$.  Therefore  if we let 
   $$A=\{ f_a | a \in P \Gamma L_n (\mathbb{F}_q)  \},   \ \ \ \ \ \ (1) $$
    then $A$ is a group isomorphic to the group $P \Gamma L_n (\mathbb{F}_q))$ (as abstract groups),  and we have $A \leq Aut(G)$.

  When $n=2k$,  then the  Grassmann graph $G=G(q,n,k)$ has some other automorphisms.  In fact if  $n=2k$,    then the
   mapping $\theta : V_k \longrightarrow V_k$, which is defined by this rule $\theta(v)=v^{\perp}$, for every $k$-subspace of $V(2k)$, is an automorphism of the graph $G=G(q,2k,k)$. 
    Hence $M=\langle A,\theta \rangle   \leq Aut(G)$. 
    It can be shown that $A$ is a normal subgroup of the group $M$. Therefore $M=A\rtimes \langle \theta \rangle  $. 
     Note that the order of $\theta$  is 2 and hence $\langle \theta \rangle   \cong \mathbb{Z}_2$.
Concerning the automorphism groups of Grassmann graphs, from a known fact which  appeared in [3],    we have the following result [2].

\begin{thm}  Let $G$ be the Grassmann graph $G=G(q,n,k)$, where $n >3$ and  $k \leq \frac{n}{2}$. If $n \neq 2k$, then we have 
$Aut(G)=A \cong P \Gamma L_n (\mathbb{F}_q)$,  and if $n=2k$, then we have $Aut(G) =\langle A, \theta \rangle  \cong A\rtimes \langle \theta \rangle  \cong  P \Gamma L_n (\mathbb{F}_q) \rtimes \mathbb{Z}_2$, where $A$ is the group which is defined in $(1)$ and $\theta$ is the mapping which   defined preceding   this theorem.
\end{thm}

We now proceed to determine the automorphism group of a class of bipartite graphs which are similar in some aspects to the graphs $B(n,k)$

                                                                                                                                  \begin{defn}
                                                                                                                                       
                                                                                                                                        Let $n,k$ be positive integers such that $n\geq 3$, $k \leq n-1 $. Let $q$ be a power of a prime and $\mathbb{F}_q$ be the finite field of order $q$. Let $V(q,n)$ be a vector space of dimension $n$ over $\mathbb{F}_q$. We define the graph $S(q,n,k)$ as  a  graph with the vertex-set $V=V_k \cup V_{k+1}$, in which two vertices $v$ and $w$ are adjacent whenever $v$ is a subspace of $w$  or $w$ is a subspace of $v$, where $V_k$ and $V_{k+1}$ are the sets of  subspaces in $V(q,n)$ of dimensions $k$ and $k+1$,  respectively.
\end{defn}                                                                                                                                         
                                                                                                                                           When $n=2k+1$, then the graph $S(q,n,k)$ is known as a doubled Grassmann graph [2]. 
                                                                                                                                           Noting that ${n\brack k}_q={n\brack n-k}_q$, it is easy to show that $S(n,q,k) \cong S(n,q,n-k-1)$. Hence in the sequel we assume $k <  \frac{n}{2}$.
                                                                                                                                         It can be shown that the graph $S(q,n,k)$ is   a connected bipartite   irreducible graph.
                                                                                                                                       We formally state and prove these facts.
                                                                                                                                       
                                                                                                                                       \begin{prop}The graph $G=S(q,n,k)$ which is defined in Definition $3.13, $  is a  connected bipartite   irreducible graph.
                                                                                                                                       \end {prop}
                                                                                                                                     
                                                                                                                                      \begin{proof} It is clear that the graph $G=S(q,n,k)$ is a bipartite graph with partition $V_k \cup V_{k+1}$. It is easy to show that $G$ is an irreducible graph.  We now show that $G$ is a connected graph. It is sufficient to show that if $v_1,v_2$ are two vertices in $V_{k}$, then there is a path in $G$ between $v_1$ and $v_2$. Let $dim(v_1 \cap v_2)=k-j$, $1 \leq j \leq k$. 
                                                                                                                                      We prove our assertion by induction on $j$. If $j=1$, then $u=v_1+v_2$ is a subspace of $V(n,q)$ of dimension $k+k-(k-1)=k+1$, which contains both of $v_1$ and $v_2$.   Hence,  $u \in V_{k+1}$ is adjacent to both of the vertices $v_1$ and $v_2$.
                                                                                                                                       Thus, if $j=1$, then there is a path between $v_1$ and $v_2$ in the graph $G$. Assume when $j=i$,  $0 < i <k$, then there is a path in $G$ between $v_1$ and $v_2$. We now assume $j=i+1$. Let $v_1 \cap v_2=w$, and let $B=\{ b_1,...,b_{k-i-1}  \}$ be a basis for the subspace $w$ in the space $V(q,n)$. We can extend $B$ to   bases $B_1$ and $B_2$ for the subspaces
                                                                                                                                        $v_1$ and $v_2$, respectively. Let $B_1= \{ b_1,...,b_{k-i-1}, c_1,...,c_{i+1}  \}$ be a basis for $v_1$ and $B_2= \{ b_1,...,b_{k-i-1}, d_1,...,d_{i+1}  \}$ be a basis for $v_2$. Consider the subspace $s=<b_1,...,b_{k-i-1}, c_1,d_2,...,d_{i+1}>$.
                                                                                                                                         Then $s$ is a $k$-subspace of the space $V(q,n)$ such that $dim(s \cap v_2)=k-1$ and $dim(s \cap v_1)=k-i$. Hence by the induction assumption, there is a path $P_1$ between vertices $v_2$ and $s$, and a path $P_2$ between vertices $s$ and $v_1$. We now conclude that there is a path in the graph $G$ between vertices  $v_1$ and $v_2$. 
                                                                                                                                         \end{proof}

\begin{thm}                                                                                                                                          Let $G=S(q,n,k)$ be the graph which is defined in Definition $3.13. $ If $n\neq 2k+1$, then we have $Aut(G) \cong P \Gamma L_n (\mathbb{F}_q)$. If $n=2k+1$, then $Aut(G) \cong P \Gamma L_n (\mathbb{F}_q) \rtimes \mathbb{Z}_2 $. 
                                                                                                                                            \end{thm}
                                                                                                                                         \begin{proof}                                                                                                                                         From Proposition 3.14,  it follows that the graph $G=S(q,n,k)$ is    connected,  bipartite and irreducible  with the vertex-set $V_k \cup V_{k+1}$, $V_k \cap V_{k+1}= \emptyset$. Let $G_1=G(q,n,k)=(V_k,E)$ be the Grassmann graph with the vertex-set $V_k$,  when $k > 1$ and $V_2$,  when $k=1$. We show that $G_1$ is  a  faithful representation of the    graph $G$.\\
                                                                                                                                           Firstly, the condition (i) of Definition 3.4,  holds  because $k < \frac{n}{2}$ and every automorphism 
of the Grassmann graph                                                                                                                                            $G(q,n,r)$  is of the form $f_a$, $a \in P \Gamma L_n (\mathbb{F}_q)$,  and  is  an automorphism of the graph $G(q,n,s)$ when $r,s < \frac{n}{2}$. Also, note that if $X,Y$ are subspaces of $V(q,n)$ such that $X \leq Y$, then $f_a(X) \leq f_a(Y)$. \\
Now, suppose that $f$ is an automorphism of the graph $G$ such that $f(V_k)=V_k$. We show that the restriction of $f$ to the set $V_k$, namely $g=f|_{V_k}$,  is an automorphism of the graph $G_1$. It is trivial that $g$ is a permutation of the vertex-set $V_k$. Let $v$ and $w$ be adjacent vertices in the graph $G_1$. We show that $g(v)$ and $ g(w)$ are adjacent in the graph $G_1$. We assert that there is exactly one vertex $u$ in the graph $G$ such that $u$ is adjacent to   both of the vertices $v$ and $w$.  If the vertex $u$ is adjacent to  both of the vertices $v$ and $w$, then $v$ and $w$ are $k$-subspaces of the $(k+1)$-space $u$. Hence 
$u$ contains the space $v+w$. Since $dim(v+w)$=$dim(v)+dim(w)-dim(v\cap w)=k+k-(k-1)=k+1$, we have  $u=v+w$. In other words, the vertex $u=v+w$ is the unique vertex in the graph $G$ such that  $u$ is adjacent to  both of the vertices $v$ and $w$. Also, note that our discussion shows  that if $x,y \in V_k$ are such that $dim(x \cap y)\neq (k-1)$, then $x$ and $y$ have no   common neighbor in the graph $G$. 
 
      Now since the vertices $v$ and $w$ have exactly 1 common neighbor in the graph $G$,  therefore $f(v)=g(v)$ and $f(w)=g(w)$ have  exactly 1 common neighbor in the graph $G$. It follows that $dim(g(v) \cap g(w))=k-1$, and hence $g(v)$ and $g(w)$ are adjacent vertices in the Grassmann graph $G_1$. \\
 We now conclude that the graph $G_1$ is  a  faithful representation  of the   graph $G$.\
 
  There are two possible cases, namely,  (1) $2k+1 \neq n$, or (2)  $2k+1 = n$.\\
  (1) Let $2k+1 \neq n$. Noting that ${n\brack k}_q < {n\brack k+1}_q$,   it follows that $|V_k| \neq |V_{k+1}|$. Therefore by Theorem 3.8,  and Theorem 3.12, we have $Aut(G) \cong Aut(G_1) \cong  P \Gamma L_n (\mathbb{F}_q)$.\\
  (2) If $2k+1=n$, since ${n\brack k}_q={n\brack k+1}_q$, 
   then $|V_k| = |V_{k+1}|$. Hence, the mapping $\theta : V(G) \longrightarrow V(G)$ defined by the rule $\theta(v) =v^{\perp}$ is an automorphism of the graph $G$ of order 2 such that
   $\theta(V_k)=V_{k+1}$. Hence, by Theorem $3.9,$ and Theorem $3.12, $  we have $Aut(G) \cong Aut(G_1) \rtimes \langle \theta  \rangle  \cong P\Gamma L_n (\mathbb{F}_q) \rtimes \mathbb{Z}_2$.

                                                                                                                                                  \end{proof}                                                                                                                                         
                                                                                                                                                We now show another application of Theorem 3.9,  in determining the automorphism groups of some classes of graphs which are important in algebraic graph theory. \

                                                                                                                                                 If $ G_1, G_2 $ are graphs, then their direct product  (or tensor product) is the graph  $   G_1 \times G_2 $
 with vertex set $ \{( v_1,v_2) \  |  \  v_1 \in G_1,   v_2 \in G_2\} $, and for which vertices $( v_1,v_2)$ and $ ( w_1,w_2)  $ are adjacent precisely if $ v_1$ is adjacent to $w_1$ in $G_1$ and $ v_2$ is adjacent to $w_2$ in $G_2$. It can be shown  that the direct product is commutative and associative [8].  The following theorem, first was proved by Weichsel
(1962),  characterizes connectedness in direct products of two factors. \

\begin{thm} $[8]$ Suppose $G_1$ and $G_2$ are connected nontrivial graphs. If at least one of $G_1$ or $G_2$ has an odd cycle, then $   G_1 \times G_2 $ is connected. If both $G_1$ and
$G_2 $ are bipartite, then $   G_1 \times G_2 $ has exactly two components.
\end{thm}   

Thus, if one of the graphs $G_1$ or $G_2$ is a connected non-bipartite graph, then the graph $   G_1 \times G_2 $ is a connected graph. 
If $K_2$ is the complete graph on the set $\{ 0,1 \}$, then the direct product $B(G)=G \times K_2$ is a bipartite graph, and is  called the bipartite double  cover of $G$ (or the $bipartite\  double$  of $G$).                                                                        
Then, 
 $$V(B(G))=\{(v,i)| v\in V(G), i \in \{ 0,1 \}  \}, $$                                                          
                                                                             and two vertices $(x,a)$ and $(y,b)$ are adjacent in the graph  $B(G)$ if and only if $a \neq b$ and $x$ is adjacent to $y$ in the graph $G$. The notion of the bipartite double cover  of $G$ has many applications in algebraic  graph theory [2]. \

                                                                             Consider the bipartite double cover of $G$, namely,  the graph $B(G)= G \times K_2 .$ It is easy to see that  the group $Aut(B(G))$ contains the group $Aut(G) \times \mathbb{Z}_2$ as a subgroup. In fact,  if for $g \in Aut(G)$,    we define the mapping $e_g$ by the rule $e_g(v,i)=(g(v),i)$, $i\in \{0,1\}, v\in V(G)$,  then $e_g \in Aut(B(G))$.   It is easy to see that $H=\{e_g | g \in Aut(G)  \} \cong Aut(G)$ is a subgroup of $Aut(B(G))$. Let $t$ be the mapping defined on $V(B(G))$ by the rule $t(v,i)=(v,i^c), $ where $i^c=1$ if $i=0$ and $i^c=0$ if $i=1$. It is clear that $t$ is an  automorphism of the graph $B(G).$ Hence, $\langle H,t \rangle  \leq Aut(B(G)).$ Noting that for every $e_g \in H$ we have $e_gt=te_g, $ we deduce that $\langle H,t \rangle  \cong H \times \langle t \rangle  $. We now conclude that $Aut(G) \times \mathbb{Z}_2 \cong H  \times \mathbb{Z}_2 \leq Aut(B(G)). $                                                                             
                                                                           \
                                                                             
                                                                             Let $G$ be a graph. $G$ is called a  $stable$ graph when we have $Aut(B(G)) \cong Aut(G) \times \mathbb{Z}_2$. Concerning the notion and some properties of stable graphs,  see [12,13,21,23].\

Let  $n,k \in \mathbb{ N}$ with $ k < \frac{n}{2}  $ and let $[n]=\{1,...,n\}$.   The $Kneser$  $graph$ $K(n,k)$ is defined as the graph whose vertex set is $V=\{v\mid v\subseteq [n], |v|=k\}$ and two vertices $v$,$w$ are adjacent if and only if $|v\cap w|$=0.                                                                                                                                                   It is easy to see that  if $H(n,k)$ is a bipartite Kneser graph, then $H(n,k) \cong K(n,k) \times K_2$. Now, it follows from Theorem 3.9, (or [18]) that   Kneser graphs are stable graphs.

  The next theorem  provides a sufficient condition such that   when a connected non-bipartite irreducible   graph $G$  satisfies this condition, then $G$ is a stable graph. 
  
\begin{thm}Let $G=(V,E)$ be a connected non-bipartite irreducible  graph. Let $v,w \in V$ be arbitrary. Let $c(v,w)$ be the number of common neighbors of $v$ and $w$ in the graph $G$. Let $a_0 >0$ be a fixed   integer. If $c(v,w)=a_0$, when $v$ and $w$ are adjacent and $c(v,w) \neq a_0$ when $v$ and $w$ are non-adjacent, then we have, 
                                                                                                 $$Aut(G \times K_2) = Aut(B(G))  \cong Aut(G) \times \mathbb{Z}_2,  $$ 
                                                                                                                  in other words, $G$ is a stable graph.
                                                                                                                                      
                                                                                                                                                                            \end{thm}

 \begin{proof} Note that the graph  $G \times K_2$ is a bipartite graph with the vertex set  $V=U \cup W$, where $U= \{ (v,0) | v\in V(G) \}$ and  $W= \{ (v,1) | v\in V(G) \}.$ Since $G$ is an irreducible  graph, then   the graph  
$G \times K_2$ is an   irreducible  graph. In fact if the vertices $x,y \in V$ are such that $N(x)=N(y)$, then $x,y \in U$ or $x,y \in W. $ Without loss of generality, we can assume that $x,y \in U.$ Let $x=(u_1,0)$ and $y=(u_2,0)$. Let $N(x)=\{(v_1,1), (v_2,1), ..., (v_m,1)  \}$ and $N(y)=\{(t_1,1), (t_2,1), ..., (t_p,1)  \}$, where $v_is$  and $t_js$ are in $V(G). $ Thus $m=p$ and $N(u_1)=\{u_1,...,u_m  \}$=$\{t_1,...,t_m  \}=N(u_2).$ Now since $G$ is an irreducible graph, it follows that $u_1=u_2$ and therefore $x=y$. \\
Let $G_1=(U,E_1)$ be the graph with vertex-set  $U$ in which two vertices $(v,0)$ and $(w,0)$ are adjacent if and only if $v_1$ and $v_2$ are adjacent in  the graph $G$.   It is clear that $G_1 \cong G$.  Therefore we have   $Aut(G_1) \cong Aut(G).$ For every   $f \in Aut(G)$,  we let

$$d_f : U \rightarrow U, \ d_f(v,0)=(f(v),0), \  for \  every \  (v,0) \in U.$$  
Then $d_f$ is an automorphism of the graph $G_1.$                  
 If we let $A=\{d_f | f \in Aut(G)  \}$,   then $A$ with the  operation of composition  is a group,
   and it is easy to see that  $A \cong Aut(G_1)$ (as abstract groups). We now assert that the graph $G_1$ is  a  faithful representation  of     the bipartite graph $B=G \times K_2$. Let $g \in Aut(B)$ be such that $g(U)=U.$
We assert that $h=g|_U$, the restriction of $g$ to $U$,  is an  automorphism of the graph $G_1$. It is clear that $h$ is a permutation of $U$. Let $(v,0)$ and $(w,0)$ be adjacent vertices in $G_1$. Then $v,w$ are adjacent in the graph $G$. Hence there are vertices $u_1,...,u_{a_0}$ in the graph $G$ such that 
the set of common neighbor(s) of $v$ and $w$ in $G$ is $\{u_1,...,u_{a_0} \}$. Noting that $(x,1)$ is a common neighbor of $(v,0)$ and $(w,0)$ in the graph $B$   if  and only if  $x$ is a common neighbor of $v,w$ in the graph $G$, we deduce that the set  $\{ (u_1,1),...,(u_{a_0},1) \}$ is the set of common neighbor(s) of $(v,0)$ and $(w,0)$ in the graph $B$. 
Since $g$ is an automorphism of the graph $B$,  $g(v,0)$ and $g(w,0)$  have $a_0$ common neighbor(s) in the graph $B$. Note that if $d_{G_1}(g(v,0),g(w,0)) >2 $, then these vertices have no common neighbor in the graph $B$. Also, if  
  $d_{G_1}(g(v,0),g(w,0)) =2,  $ then $d_G(v,w)=2$, and hence $v,w$ have $  c(v,w)\neq a_0$ common neighbor(s)  in the graph $G$. Hence $(v,0)$ and $(w,0)$ have $c(v,w)$ common neighbor(s) in the graph $B$, and therefore  $g(v,0),g(w,0)$ have $  c(v,w) \neq a_0$ common neighbor(s) in the graph $B$. We now deduce that $d_{G_1}(g(v,0),g(w,0)) =1$. It follows that $h=g|_U$ is an automorphism of the graph $G_1$. 
  Thus, the condition (ii) of Definition 3.4, 
     holds for the graph $G_1.$\

Now, suppose that $\phi$ is an automorphism of the graph $G_1.$ Then there is an automorphism $f$ of the graph $G$ such that $\phi = d_f.$ We now define   the mapping $e_{\phi}$ on the set $V(B)$ by the following rule:
 
 $$ (*)\ \ \ \  e_{\phi}(v,i) =  \begin{cases}
(f(v),0), \   $ if$ \    i=0  \\   (f(v),1),       \ \       $if$ \  i=1.  \\
 \end{cases} $$ \\
It is easy to see that $e_{\phi}$ is an extension  of the automorphism $\phi$ to an automorphism of the graph $B$.   We now deduce that the graph $G_1$ is  a  faithful representation of   the graph $B$. \\
 On the other hand, it is easy to see that the mapping $t : V(B) \rightarrow V(B)$, which is defined by the rule,  
$$(**) \ \ \ \  t(v,i) =  \begin{cases}
(v,0),  \ $if$ \    i=1  \\   (v,1),  \ $if$ \  i=0,  \\
 \end{cases} $$ \\     
is an automorphism of the graph $B$ of order 2. Hence, $\langle t \rangle \cong \mathbb{Z}_2. $ Also, it is easy to see that for every automorphism $\phi$ of the graph $G_1$ we have $te_{\phi}=e_{\phi}t$.   We now conclude by Theorem 3.9,    that 
 $$Aut(G \times K_2)=Aut(B) \cong Aut(G_1) \rtimes \langle t  \rangle  \cong Aut(G) \times \langle t  \rangle \cong Aut(G) \times \mathbb{Z}_2.$$

\end{proof}              
                                                                                                                                                  
                                                                                                                                                  As an application of Theorem 3.17,   we show that the Johnson graph $J(n,k)$ is a stable graph. Since $J(n,k) \cong J(n,n-k)$, in the sequel we assume that $k \leq \frac{n}{2}$. 
                                                                                                                                                    
                                                                                                                                                    \begin{thm} Let $n,k$ be positive integers with $k \leq \frac{n}{2}$. If $n\neq 6$,  then the Johnson graph $J(n,k)$ is a stable graph.

\end{thm}                                                                                                                                                           

\begin{proof} We know that the vertex set of the graph $J(n,k)$ is the set of $k$-subsets of $[n]=\{ 1,2,3,...,n \}$ in which two vertices $v$ and $w$ are adjacent if and only if $|v\cap w|=k-1$.   If $k=1$, then $J(n,k) \cong K_n$, the complete graph on $n$ vertices. It is easy to see that if $X=K_n$,  then the bipartite double cover of $X$ is isomorphic with the bipartite Kneser graph $H(n,1)$. From Theorem 3.9 (or [18]),  we know that $Aut(H(n,1)) \cong Sym([n]) \times \mathbb{Z}_2 \cong Aut(K_n) \times  \mathbb{Z}_2$. Hence the Johnson graph $J(n,k)$ is a stable graph when $k=1$. We now assume that $k \geq 2$. We let $G=J(n,k)$. It is easy to see that 
$G$ is an irreducible graph.
 It can be shown that if $v,w$ are vertices in $G$, then 
$d(v,w)=k-|v \cap w|$ [2]. Hence,  $G$ is a connected graph. 
It is easy to see that the girth of the Johnson graph $J(n,k)$ is 3.  Therefore, $G$ is a non-bipartite graph. 
  It is clear that when $d(v,w) \geq 3$, then $v,w$ have no common neighbors. We now consider two other possible cases, that is, (i) $d(v,w)=2$ or (ii) $d(v,w)=1$. Let $c(v,w)$ denote  the number of common neighbors of $v,w$ in $G$. In the sequel, we show that if $d(v,w)=2$, then $c(v,w)=4$, and if $d(v,w)=1$, then $c(v,w)=n-2$. \

 (i) If $d(v,w)=2$, then $|v\cap w|=k-2$. Let $v \cap w=u$. Then $v=u \cup \{i_1,i_2  \}$, 
$w=u \cup \{j_1,j_2  \}$, where $i_1,i_2,j_1,j_2 \in [n]$, $\{i_1,i_2  \} \cap \{j_1,j_2  \}=\emptyset$. Let $x \in V(G)$. It is easy to see that if $|x\cap u| < k-2$, then $x$ can not be a common neighbor of 
$v,w$. Hence, if $x$ is a common neighbor of $v$ and $w$,  then $x$ is of the form $x=u\cup \{r,s  \}$, where $r \in \{ i_1,i_2 \}$ and $s \in \{ j_1,j_2 \}$. We now deduce that the number of common neighbors of $v$ and $w$ in the graph $G$ is 4. \

(ii) We now assume that $d(v,w)=1$. Then $|v\cap w|=k-1$. Let $v \cap w=u$. Then $v=u \cup \{r\}$, 
$w=u \cup \{s \}$, where $r,s \in [n]$, $r \neq s$. Let $x \in V(G)$. It is easy to see that if $|x\cap u| < k-2$, then $x$ can not be a common neighbor of 
$v,w$. Hence, if $x$ is a common neighbor of $v$  and  $w$,  then $|x \cap u|=k-1$ or $|x \cap u|=k-2$. In the first step, we assume that  $|x \cap u|=k-1$. Then $x$ is of the form $x=u\cup \{ y\}$, where $y \in [n]-(v\cup w)$. Since, $|v\cup w|=k+1$, then the number of such $x\rq{}s$ is $n-k-1$.\newline
We now assume that  $|x \cap u|=k-2$. Hence, $x$ is of the form $x=t \cup \{ r,s \}$, where $t$ is a $(k-2)$-subset of the $(k-1)$-set $u$. Therefore the number of such $x\rq{}s$ is ${k-1} \choose{k-2}$=$k-1$.\
Our argument follows that if $v$ and $w$ are adjacent, then we have $c(v,w)=n-k-1+k-1=n-2$. \

 Noting that $n-2 \neq 4$, we conclude from Theorem 3.18 that the Johnson graph $J(n,k)$ is a stable graph  when $n \neq 6$.

\end{proof}
Although, Theorem 3.18,  does not say anything about the stability of the Johnson graph $J(6,k)$, we show by the next result that this graph is a stable graph.                                                                                                                                                             
                                                                                                                                                             
                                                                                                                                                             \begin{prop} The Johnson graph $J(6,k)$ is a stable graph.

                                                                                                                                                                \end{prop}
                                                                                                                                                                 
                                                                                                                                                                 \begin{proof} When $k=1$ the assertion is true, and  hence we assume that $k\in \{ 2,3 \}$. In the first step we show that the  Johnson graph $J(6,2)$ is a stable graph. Let $B=J(6,2) \times K_2$. We show that $Aut(B) \cong Sym([6]) \times \mathbb{Z}_2$, 
                                                                                                                                                                 where $[6]= \{ 1,2,...,6  \}$. It is clear that  $B$ is a bipartite irreducible   graph. Let $V=V(B)$ be the vertex-set of the graph $B$. Then $V=V_0 \cup V_1$, where $V_i= \{(v,i) | v\subset [6], |v|=2  \}$, $i \in \{ 0,1\}$. Let $G_1=(V_0,E_1)$ be the graph with the vertex-set $V_0$  in which two vertices $(v,0),(w,0)$ are adjacent whenever $|v \cap w|=1$. It is clear that $G_1$ is isomorphic with the Johnson graph $J(6,2)$. Hence,    we have $Aut(G_1) \cong Sym([6])$.                                                                                                                                                             
                                                                                                                                                                 We show that $G_1$ is  a  faithful representation  of    the graph $B$. By what is seen in (*) of  the proof of Theorem 3.17,  it is clear that if $h$ is an automorphism of the graph $G_1$, then $h$ can be extended to an automorphism $e_h$ of the graph $B$. Thus, the condition (i) of Definition 3.4,  holds for the graph $G_1$.  
                                                                                                                                                                 
                                                                                                                                                                   Let $a=(v,0)$ and $b=(w,0)$ be two  adjacent vertices in  the 
                                                                                                                                                                 graph $G_1$,    that is,  $|v\cap w|=1$. Let $N(a,b)$ denote  the set of common neighbors of $a$ and $b$ in the graph $B$. Let $X(a,b)=\{a,b\}\cup N(a,b) \cup t(N(a,b))$, where $t$ is the automorphism of the graph $B$  defined by the rule $t(v,i)=(v,i^c)$, $i^c \in \{ 0,1  \}, i^c\neq i$. Let $\langle X(a,b)  \rangle $ be the subgraph induced by the set $X(a,b)$ in the graph $B$. It can be shown that if 
                                                                                                                                                                  $a,b$ are adjacent vertices in $G_1$, that is,  $|v\cap w|=1$, then $ \langle X(a,b) \rangle $ has a vertex of degree $0$. On the other hand, when $a,b$ are  not 
                                                                                                                                                                 adjacent vertices in $G_1$, that is,  $|v\cap w|=0$, then $\langle X(a,b) \rangle$ has no vertices of degree $0$. In the
                                                                                                                                                                 rest of the proof, we let $\{ x,y \}=xy$.  For example, let $r=(12,0)$ and $s=(13,0)$ be two 
                                                                                                                                                                 adjacent vertices of $G_1$. Then $X(r,s)=\{(12,0),(13,0),(14,1),(15,1),(16,1),(23,1),(14,0),(15,0),(16,0),(23,0)  \}$. Now,  in the graph $\langle X(r,s)  \rangle$ the vertex $(23,0)$ is a vertex of degree 0. Whereas, if we let 
$r=(12,0)$, $u=(34,0)$, then $r,u$ are not adjacent in the graph $G_1$. Then $X(r,u)=\{ (12,0),(34,0),(13,1),(14,1),(23,1),(24,1),(13,0),(14,0),(23,0),(24,0)\}$. Now, it is clear that the graph $\langle X(r,u) \rangle $ has no  vertices of degree 0. \\
 Note that the graph $G_1$ is isomorphic with the Johnson graph $J(6,2)$,  and hence $G_1$ is a distance-transitive graph.
    Now if $c,d$ are  two adjacent vertices in the graph $G_1$,  then there is an automorphism $f$ in $Aut(G_1)$ such that $f(r)=c$ and $f(s)=d$. Let $e_f$ be the extension of $f$ to an automorphism of the graph $B$.
     Therefore,  $\langle X(c,d)  \rangle = \langle X(e_f(r),e_f(s)) \rangle =e_f(\langle X(r,s)\rangle)$ has   a vertex  of degree $0$.
      This argument also shows that if $p,q$ are non-adjacent vertices in the graph $G_1$, then $\langle X(p,q) \rangle$ has no vertices of degree $0$.\

Now, let $g$ be an automorphism of the graph $B$ such that $g(V_0)=V_0$. We show that $g|_{V_0}$ is an automorphism of the graph $G_1$. Let $a=(v,0)$ and $b=(w,0)$ be two adjacent  vertices of the 
                                                                                                                                                                 graph $G_1$,  that is,  $|v\cap w|=1$. Then $\langle X(a,b) \rangle$ has a vertex of degree 0. Hence,  $g(\langle X(a,b)\rangle)=\langle X(g(a),g(b))\rangle$ has a vertex of degree 0. Then $g(a)$ and $g(b)$ are adjacent in the graph $G_1$. We now deduce that if 
                                                                                                                                                                 $g$ is an automorphism of the graph $B$ such that $g(V_0)=V_0$, then $g|_{V_0}$ is an automorphism of the graph $G_1$. Therefore, the condition (ii) of Definition 3.4,   holds for the graph $G_1$.
                                                                                                                                                                  Therefore, $G_1$ is  a  faithful representation  of the    graph $B$.   Note that $t$ is an automorphism of the graph $B$ of order 2 such that $t(V_0)=V_1$ and $t(V_1)=V_0$. Also, we have $tf=ft, f\in Aut(G_1)$.   We now, conclude by Theorem 3.9, that; 
                                                                                                                                                                 $$Aut(B) \cong Aut(G_1)\rtimes \langle t \rangle  \cong Aut(G_1)\times \langle t \rangle  \cong  Aut(G)\times \mathbb{Z}_2 \cong Sym([6])\times \mathbb{Z}_2 $$ 
                                                                                                                                                                  Therefore, the graph $G=J(6,2)$ is a stable graph.\
                                                                                                                                                                  
                                                                                                                                                                  By a similar argument, we can show that the graph $J(6,3)$ is a stable graph.
                                                                                                                                                                   \end{proof}
                                                                                                                                                                    Combining Theorem 3.18,   and Proposition 3.19,  we obtain the following result.
                                                                                                                                                                    \begin{thm} The Johnson graph $J(n,k)$ is a stable graph.

                                                                                                                                                                     \end{thm}

                                                                                                                                                                       \section{Conclusion}
                                                                                                                                      In this paper, we gave a method for finding the automorphism groups  of  connected bipartite  irreducible  graphs (Theorem 3.8, and Theorem 3.9).  Then  by our method, we explicitly
 determined  the automorphism groups of some classes of bipartite irreducible graphs, including the graph $S(q,n,k)$ which is a derived graph from the Grassman graph $G(q,n,k)$ (Theorem 3.15). Also, we provided a sufficient ascertainable condition such that  when
  a connected non-bipartite irreducible  graph $G$ satisfies this condition, then $G$ is a  stable graph (Theorem 3.17). 
                                                                                                                                      Finally, we showed that the Johnson graph $J(n,k)$ is a stable graph (Theorem 3.20).
 
 \section{ Acknowledgements}
The author is thankful to the anonymous referee  for his/her valuable comments and suggestions.

\end{document}